\documentclass[11pt]{article}%
\usepackage{amsmath}
\usepackage[table]{xcolor}
\usepackage{amsfonts}
\usepackage{amssymb}
\usepackage{graphicx}%
\providecommand{\U}[1]{\protect\rule{.1in}{.1in}}
\newtheorem{theorem}{Theorem}

\newtheorem{conjecture}[theorem]{Conjecture}

\newtheorem{example}[theorem]{Example}

\newtheorem{lemma}[theorem]{Lemma}

\newtheorem{proposition}[theorem]{Proposition}

\newenvironment{proof}[1][Proof]{\noindent\textbf{#1.} }{\ \rule{0.5em}{0.5em}}
\newenvironment{proofm}[1][Proof of Theorem \ref{te-1}]{\noindent\textbf{#1.} }{\ \rule{0.5em}{0.5em}}

\begin{document}

\title{Matrices in  ${\cal A}(R,S)$ with minimum  $t$-term ranks}
\author{Ros\'{a}rio Fernandes\thanks{Corresponding author.\,Email: mrff@fct.unl.pt. This work was partially supported by project UID/MAT/00297/2019.}\\CMA and Faculdade de Ci\^{e}ncias e Tecnologia\hfill\\Universidade Nova de Lisboa\\2829-516 Caparica, Portugal. \and Henrique
F. da Cruz\thanks{Email: hcruz@ubi.pt.This work was partially supported by project UID/MAT/00212/2019.}\\
Departamento de Matem\'atica da Universidade da Beira
Interior,\hfill \\Rua Marqu\^es D'Avila e Bolama, 6201-001
Covilh\~a, Portugal.
\and Susana C. Palheira\thanks{Email: susanap1112@gmail.com.}\\Faculdade de Ci\^{e}ncias e Tecnologia\hfill\\Universidade Nova de Lisboa\\2829-516 Caparica, Portugal.}
\maketitle

\begin{abstract}  Let $R$ and $S$ be two sequences of nonnegative integers in nonincreasing order and with the same sum, and let  ${\cal A}(R,S)$ be the class  of all $(0,1)$-matrices having row sum $R$ and column sum $S$. For a  positive integer $t$, the $t$-term rank of a $(0,1)$-matrix $A$ is defined as the maximum number of $1$'s in $A$ with at most one $1$ in each column, and at most $t$ $1$'s in each row. In this paper, we address conditions for the existence of a  matrix in a class   ${\cal A}(R,S)$ that realizes all  the minimum $t$-term ranks, for  $t\geq 1$.

\end{abstract}

\vspace{0,4cm}

{\it Keywords}: t-term rank,  (0,1)-matrix, Gale-Ryser Theorem, Network flows

\vspace{0,2cm}

{\it AMS Subject of Classification}: 05A15, 05C50, 05D15

\section{Introduction}

\hspace{3ex}Consider the following problem: ``We are organizing a  social dinner for the participants of a mathematical meeting.  The participants are from   $n$ different countries. We intent to seat the participants  in $m$ tables, each one with a fix number of seats, so that in each table there are  no two participants with the same nationality. The number of seats is equal to the number of participants and there is a possible distribution in these conditions.
After dinner, there will be a show. Before the show, the host will call some participants for a joke. The host will choose  the maximum number of participants selecting, at most, one per table and, at most, one per country. As we do not know when the artists will be ready for the show, the host will repeat the joke with another group, with the maximum number of participants, again choosing, at most, one  per table and one per country, but the  countries  are now different from the   countries  selected in the first group.  The host  will repeat this joke until the show begins, or all countries were chosen.
How to distribute the participants  in the tables so that after $t$ rounds of jokes the total number of evolved participants  is as small as possible?"

The above problem is a  classical problem of distributing  $n$ elements into $m$ sets with some constraints. An important tool for solving these kind of problems are matrices  whose entries are just $0$'s and $1$'s, the $(0,1)$-matrices. In fact,  Ford and Fulkerson noted   that   a $(0,1)$-matrix  can be regarded as distributing  $n$ elements into $m$ sets: the $1$'s in row $i$ designate the elements that occur in the $i$th set, and the $1$'s in column $j$ designate the sets that contain the $j$th element (see  \cite{ford}). Therefore, the $(0,1)$-matrices are essential tools in many combinatorial investigations, and hence these matrices are among those who have been received more attention in the last years (see \cite{B2,B3,B-14,fer2,fer1}. Using a $(0,1)$-matrices for  modeling a certain distribution of the participants into the tables, in our problem, the number of selected participants after $t$ rounds is a combinatorial parameter (properties which are invariant under arbitrary permutations of its rows and  columns) of a $(0,1)$-matrix called {\it $t$-term rank}. The  $t$-term rank of $(0,1)$-matrix $A$, denoted $\rho_t (A)$, is the maximum number of $1^{\prime}$s in $A$ with at most one $1$ in each column and at most $t$ $1^{\prime}$s in each row. When $t=1$, we have the well-known term rank of $A$, denoted $\rho (A)$. In terms of the incidence matrix of sets vs. elements, the $t$-term rank of a $(0,1)$-matrix is  the maximum number of distinct elements that we have if we choose at most $t$ elements in each set, associated. So a solution for the above problem is a $(0,1)$-matrix with prescribed row and column sum vectors, that realizes all the minimum $t$-term ranks, for  $t>1$. In this paper, we focus our attention on this kind  of $(0,1)$-matrices, that is,   $(0,1)$-matrices with prescribed row and columns sum vectors with all the minimum  $t$-term ranks, for  $t>1$.

\section{Background on $\mathcal{A}(R,S)$\label{back}}

\hspace{3ex}Let $R=(R_1,\ldots ,R_m)$, and $S=(S_1,\ldots ,S_n)$ be two partitions of the same weight (this is, $R$ and $S$ are integral vectors such that $R_1\geq\ldots\geq R_m>0$, $S_1\geq \ldots\geq S_n>0$ and $R_1+\ldots +R_m=S_1+\ldots +S_n$). The class of all $(0,1)$-matrices with row sum vector $R$ and column sum vector $S$ is denoted by ${\cal A}(R,S)$. This class  has been heavily investigated (see \cite{B2,B3,B5} for details) since the fifties, and many notable results have been obtained. The aim of this section is to make a brief sketch of  these  results.

A question that first arise when we study a class ${\cal A}(R,S)$ is to know when it is nonempty. This problem was solved independently by Gale (see \cite{G}) and by Ryser (see \cite{Ry}). In fact, they proved a theorem, now  called the {\it  Gale-Ryser theorem}, which states that  the class ${\cal A}(R,S)$ is nonempty if and only if $S$ is majorized  by $R^*=(R^*_1,\ldots ,R^*_{R_1})$ (the conjugate partition of $R$, defined by $R_{j}^{\ast}=|\{i:\ m\geq i\geq 1,\ R_{i}\geq j\}|$, for $j=1,\ldots ,R_1$), i.e.

$$\sum_{i=1}^k S_i \leq \sum_{i=1}^k R^*_i,\,\,\,\,\,\,\,k=1,\ldots,n,$$

\noindent where the equality holds for $k=n$.

 In the Ryser's paper, \cite{Ry},  an algorithm for the construction of a matrix in ${\cal A}(R,S)$ is also presented. This algorithm, called {\it  the Ryser's algorithm}, starts with a $m$-by-$n$ $(0,1)$-matrix $\overline{A}$ whose row sum vector is $R$ and whose column sum vector is the conjugate vector $R^*$. Thus the 1's occupy the initial positions in each row. The construction begins by shifting $S_n$ of the last 1's of certain rows of $\overline{A}$ to column $n$. The 1's in column $n$ are in rows of  $\overline{A}$ with largest sum, giving preference to the bottommost positions in case of ties. Reducing by 1 those $R_i$ corresponding to the rows that contains the 1's placed in column $n$, we obtain a vector $R'$ which also satisfies the monotonicity assumption. We now proceed inductively to construct the columns $n-1,n-2,\ldots,1$. Ryser proved that if  $S$ is majorized  by $R^*$, then this algorithm can be  carried out in order produce a matrix in  ${\cal A}(R,S)$. When ${\cal A}(R,S)\neq \emptyset$ the matrix $\widetilde{A}$ constructed by Ryser's algorithm is often called the {\it canonical matrix} of  ${\cal A}(R,S)$.

 \vspace{0,1cm}

 Another fundamental result due to Ryser is the so called {\it Ryser's interchange theorem}, \cite{B3}, which states  that given two matrices $A,B \in {\cal A}(R,S)$, $A$ can be transformed in $B$ by a finite sequence of {\it interchanges}.  An interchange is an operation which replaces a $2$-by-$2$ submatrix $$\left[\begin{array}{cc}1&0\\ 0&1\end{array}\right] \hspace{4ex}\mbox{ into }\hspace{4ex}\left[\begin{array}{cc}0&1\\ 1&0\end{array}\right],$$ and vice-versa. Therefore, every matrix in ${\cal A}(R,S)$ can be transformed by interchanges into the canonical matrix   $\widetilde{A}$  of  ${\cal A}(R,S)$.

 \vspace{0,1cm}

 Ford and Fulkerson, presented another criterium for the nonemptiness  of ${\cal A}(R,S)$. Let $T(R,S)=T=[t_{i,j}]$ be the $(m+1)$-by-$(n+1)$ matrix (the rows are indexed in $\{0,\ldots,m\}$ and the columns are indexed in $\{0,\ldots,n\}$), such that
$$t_{k,l}=kl-\sum_{j=1}^l S_j + \sum_{i=k+1}^m R_i,\,\,\,\,(\mbox{for }0\leq k\leq m,\,\,\,\,0\leq j\leq n).$$
This matrix is called the {\it structure matrix} associated  with $R$ and $S$, and Ford and Fulkerson proved that ${\cal A}(R,S)\neq \emptyset$ if and only if each entry of $T$ is nonnegative.

Let $t$ be a positive integer. Motivated by the study of combinatorial batch codes (see \cite{B4,Cod1,P}), the authors of \cite{B5} defined the  $t$-term rank of a $(0,1)$-matrix $A$, denoted $\rho_t (A)$, as the maximum number of $1^{\prime}$s in $A$ with at most one $1$ in each column and at most $t$ $1^{\prime}$s in each row. When $t=1$, we have the term rank of $A$, denoted $\rho (A)$.

By the  K$\ddot{o}$nig-Egerv$\acute{a}$ry theorem (see \cite{B3}, p.6) $\rho (A)$ equals the minimum number of lines that cover all the $1^{\prime}$s of $A$:
$$\rho (A)=\min\{e+f:\,\, \exists \mbox{ a cover of }A\mbox{ with } e\mbox{ rows and }f \mbox{ columns}\}.$$

In \cite{B5} the authors established a generalization of this theorem:

\begin{proposition}\label{te+f} {\rm \cite{B5}}  Let $A$ be an $m$-by-$n$, $(0,1)$-matrix and let $t$ be a positive integer. Then
$$\rho_t (A)=\min\{te+f:\ \exists \mbox{ a cover of }A\mbox{ with } e\mbox{ rows and }f \mbox{ columns}\}.$$

\end{proposition}

\vspace{2ex}

The term rank and the $t$-term rank are two of several combinatorial parameters  of a $(0,1)$-matrix. In this paper we turn our attention to the minimal $t$-term rank of a nonempty class ${\cal A}(R,S)$, denoted by $\tilde{\rho_t}(R,S)$. So, $$\tilde{\rho_t}(R,S)=\min \{\rho_t(A):\ A\in {\cal A}(R,S)\}.$$

A formula for computing the $\tilde{\rho_1}(R,S)$ was first derived by Haber and simplified by Brualdi \cite{B3}. A generalization for $t>1$ was stated by Fernandes and da Cruz in \cite{fer}. This formula is obtained using a matrix $\Phi$ defined as follows: Let $[t_{i,j}]$  be  the $(m+1)$-by-$(n+1)$ structure matrix associated with $R$ and $S$. Define the matrix $\Phi=[\phi_{k,l}]$,  also denoted by $\Phi(R,S)$  by

$$\phi_{k,l}=\min\{t_{i_1,l+j_2}+t_{k+i_2,j_1}+(k-i_1)(l-j_1)\},$$

\noindent  for all $0\leq k\leq m,$  $0\leq l\leq n$, and the minimum is taken  over all integers   $i_1,i_2$ and $j_1,j_2$  that satisfy
$$0\leq i_1\leq k\leq k+i_2\leq m,\,\,\,\,\text{ and }\,\,\,\,0\leq j_1\leq  l\leq  l+j_2\leq n.$$

\begin{proposition} {\rm\cite{fer}} \label{t} Let $R$ and $S$ be partitions of the same weight such that the class  $\mathcal{A}(R,S)$ is nonempty. Let $t$ be a positive integer. Then,
	$$ \tilde{\rho}_t(R,S)= \min \{te+f : 0\leq e \leq m,\,\,\,0 \leq f\leq n,\,\,\,\phi_{e,f}=t_{e,f}\}.$$
	
\end{proposition}

\vspace{2ex}
In \cite{fer} the authors proved the following theorem which, in particular, implies the existence of a special matrix in $\mathcal{A}(R,S)$ with $t$-term rank $ \tilde{\rho}_t(R,S)$.

\begin{theorem}\label{t3} {\rm \cite{fer}} Let $R$ and $S$ be partitions of the same weight such that the class  $\mathcal{A}(R,S)$ is nonempty. Then, there is a matrix in $\mathcal{A}(R,S)$ all of whose $1'$s are contained in the union of its first $e$ rows and first $f$ columns if and only if $\phi_{e,f} = t_{e,f}$
\end{theorem}

 Our  main purpose in this paper is to know if in any nonempty class ${\cal A}(R,S)$ there is a matrix  that realizes all  the minimum $t$-term ranks, for $t\geq 1$. The analogous problem for the maximum $t$-term rank in ${\cal A}(R,S)$ was solved in   \cite{B5}. In fact, the authors proved that in any nonempty class  $  {\cal A}(R,S)$ there always exists a matrix $A$ which realizes  all  the maximum  $t$-term ranks, for $t\geq 1$, and the authors of \cite{fer} conjectured that the same happens for all the minimum  $t$-term ranks, with $t\geq 1$.

 \begin{conjecture}{\rm \cite{fer}}\label{c}
 	If $R$ and $S$ are partitions of the same  weigh such that ${\cal A}(R,S)$ is nonempty, then there is a matrix $A \in {\cal A}(R,S)$ such that  $\rho_t(A)=\tilde{\rho}_t(R,S)$, for all $t\geq 1$
 \end{conjecture}

 This paper is organized as follows: As we wrote in the Introduction, the $(0,1)$-matrices can be regarded as incidence matrices of sets vs. elements, so, in the next section we compute the $t$-term rank of a matrix using the Ford-Fulkerson algorithm. Despite what happens with the maximum $t$-term rank, in the third section we present a class $ {\cal A}(R,S)$ where there is no matrix $A$ that realizes all the minimum  $t$-term ranks, for $t\geq 1$. However with some restrictions on $R$ and $S$ it is possible to prove that there are classes ${\cal A}(R,S)$ where  such matrices exist. This is the subject of the fourth section. The existence of these classes  depends of the existence of a special matrix with prescribed zero blocks. In the last section we present an algorithm for construct $(0,1)$-matrices with these constraints .

 \section{Network flows and the t-term rank}

\hspace{3ex}Another way to obtain the $t$-term rank of a $(0,1)$-matrix $A=[a_{ij}]$ is using the network flows. Let $G=(X\cup Y, U)$ be the bipartite direct graph where $X=\{1,\ldots ,m\}$, $Y=\{1,\ldots ,n\}$ and there is the edge $(x,y)$ of $G$, with $x\in X$ and $y\in Y$, if and only if $a_{xy}=1$.

Let $G_0=(Z_0,U_0)$ be the graph obtained from $G$ putting two new vertices, $s$ and $w$, and the edges $(s,p)$, $(f,w)$, for $p=1,\ldots ,m$, $f=1,\ldots ,n$.

Let $v_0$ be the map from $U_0$ to $\{0,1,\ldots ,t\}$ such that $$v_0((c,d))=\left\{\begin{array}{ll}
t & \mbox{ if }c=s\\
1 & \mbox{ otherwise }
\end{array}\right..$$

The $t$-term rank can also be obtained using the Ford-Fulkerson algorithm \cite{ford}:
\begin{enumerate}
\item Let $i=0$.
\item If there is a path $P_i$, in $G_0$, from $s$ to $w$, $$s,(s,x_{i_1}), x_{i_1}, (x_{i_1},y_{i_1}),y_{i_1}, (x_{i_2},y_{i_1}),x_{i_2},\ldots ,y_{i_h},(y_{i_h},w),w,$$ such that $$v_i((s,x_{i_1}))>0,$$ $$v_i((x_{i_f}, y_{i_f}))=1,\hspace{3ex}\mbox{ for } f=1,\ldots ,h,$$ $$v_i(( x_{i_{f+1}},y_{i_f}))=0,\hspace{3ex}\mbox{ for } f=1,\ldots ,h-1,$$ $$v_i(( y_{i_h},w))=1,$$ then go to $3.$. Otherwise, go to $4.$.
    \item Let $v_{i+1}$ be the map from $U_0$ to $\{0,1,\ldots ,t\}$ such that $$v_{i+1}((c,d))=\left\{\begin{array}{ll}
v_i((c,d))-1 & \mbox{ if $(c,d)\in P_i$ and if there is $l\in\{1,\ldots ,h-1\},$}\\
& \mbox{such that $c=x_{i_{l+1}}$, then $d\neq y_{i_l}$} \\
1 & \mbox{ if $(c,d)\in P_i$ and  there is $l\in\{1,\ldots ,h-1\},$}\\
& \mbox{such that $c=x_{i_{l+1}}$ and $d= y_{i_l}$} \\
v_i((c,d)) & \mbox{ otherwise }
\end{array}\right.$$ Go to $2.$ with $i+1$.
\item Stop. $\rho_t(A)=i$.
\end{enumerate}

\begin{example} Let $A=\left[\begin{array}{cccc}1&1&1&1\\ 1&0&0&0\\ 1&0&0&0\end{array}\right].$ Using the algorithm we can obtain $\rho_2(A)$ as described:

Let $X=\{1,2,3\}$ be the set of rows of $A$ and $Y=\{a,b,c,d\}$ be the set of columns of $A$. Let $G_0$ be the direct graph whose set of vertices, $Z_0$, is $\{s\}\cup X\cup Y\cup\{w\}$ and set of edges is $U_0=\{(s,1),(s,2),(s,3),(1,a),(1,b),(1,c),(1,d),(2,a),(3,a)\}\cup$ $\{(a,w),(b,w),(c,w),(d,w)\}$.

Let $v_0$ be the map from $U_0$ to $\{0,1,2\}$ such that $$v_0((x,y))=\left\{\begin{array}{ll}
2 & \mbox{ if }x=s\\
1 & \mbox{ otherwise }
\end{array}\right..$$
We begin the  algorithm with $i=0$ and the path $$s, (s,1), 1,(1,b),b,(b,w),w.$$

Let $v_1$ be the map from $U_0$ to $\{0,1,2\}$ such that $$v_1((x,y))=\left\{\begin{array}{ll}
2 & \mbox{ if }x=s \mbox{ and }y\neq 1\\
0 & \mbox{ if }(x,y)=(1,b) \mbox{ and }(x,y)=(b,w)\\
1 & \mbox{ otherwise }
\end{array}\right..$$

With $i=1$, the path $$s, (s,1), 1,(1,c),c,(c,w),w$$ is a path in the conditions of step {\rm 2.}, let $v_2$ be the map from $U_0$ to $\{0,1,2\}$ such that $$v_2((x,y))=\left\{\begin{array}{ll}
2 & \mbox{ if }x=s \mbox{ and }y\neq 1\\
0 & \mbox{ if }(x,y)\in\{(1,b),(1,c),(b,w),(c,w),(s,1)\}\\
1 & \mbox{ otherwise }
\end{array}\right..$$

With $i=2$, the path $$s, (s,2), 2,(2,a),a,(a,w),w$$ is a path in the conditions of step {\rm 2.}, let $v_3$ be the map from $U_0$ to $\{0,1,2\}$ such that $$v_3((x,y))=\left\{\begin{array}{ll}
2 & \mbox{ if }(x,y)=(s,3)\\
0 & \mbox{ if }(x,y)\in\{(1,b),(1,c),(b,w),(c,w),(s,1),(2,a),(a,w)\}\\
1 & \mbox{ otherwise }
\end{array}\right..$$

Since with $i=3$ there is any path in the conditions of step {\rm 2.}, then $$\rho_2(A)=i=3.$$

\end{example}

\vspace{2ex}

 \section{A counterexample for conjecture \ref{c}}

\hspace{3ex} Let $R=(6,5,4,3,3,2,2,1,1)$, and  $S=(7,3,3,2,2,1,1,1,1,1,1,1,1,1,1).$

The structure matrix associated with $R$ and $S$ is

\vspace{2ex}

$$\small T=\begin{array}{cc} & \begin{array}{cccccccccccccccc}0 \mbox{ }& 1\mbox{ }\mbox{ }& 2\mbox{ }& 3\mbox{ }\mbox{ }& 4\mbox{ } &5\mbox{ }&6\mbox{ }&7\mbox{ }\mbox{ }&8\mbox{ }&9\mbox{ }&10&11&12&13\mbox{ }&14\mbox{ }&15\end{array} \\ \begin{array}{c} 0\\ 1\\ 2\\ 3\\ 4\\ 5\\ 6\\ 7\\ 8\\ 9\end{array}&\left[\begin{array}{cccccccccccccccc}27&20&17&14&12&10&9&8&7&6&5&4&3&2&1&\cellcolor[gray]{0.8}0\\ 21&15&13&11&10&9&9&9&9&\cellcolor[gray]{0.8}9&\cellcolor[gray]{0.8}9&\cellcolor[gray]{0.8}\cellcolor[gray]{0.8}9&\cellcolor[gray]{0.8}9&\cellcolor[gray]{0.8}9&\cellcolor[gray]{0.8}9&9\\ 16&11&10&9&9&\cellcolor[gray]{0.8}9&\cellcolor[gray]{0.8}10&\cellcolor[gray]{0.8}11&\cellcolor[gray]{0.8}12&13&14&15&16&17&18&19\\ 12&8&8&\cellcolor[gray]{0.8}8&\cellcolor[gray]{0.8}9&10&12&14&16&18&20&22&24&26&28&30\\ 9&6&7&8&10&12&15&18&21&24&27&30&33&36&39&42\\ 6&4&\cellcolor[gray]{0.8}6&8&11&14&18&22&26&30&34&38&42&46&50&54\\ 4&3&6&9&13&17&22&27&32&37&42&47&52&57&62&67\\ 2&\cellcolor[gray]{0.8}2&6&10&15&20&26&32&38&44&50&56&62&68&74&80\\ 1&2&7&12&18&24&31&38&45&52&59&66&73&80&87&94\\ \cellcolor[gray]{0.8}0&2&8&14&21&28&36&44&52&60&68&76&84&92&100&108\end{array}\right]\end{array}$$
 and the matrix $\Phi$ is
$$\small \Phi =\begin{array}{cc} & \begin{array}{cccccccccccccccc}0 & 1& 2& 3\mbox{ }\mbox{ }& 4\mbox{ }\mbox{ } &5\mbox{ }&6\mbox{ }\mbox{ }&7\mbox{ }&8\mbox{ }&9\mbox{ }&10&11&12\mbox{ }&13&14\mbox{ }&15\end{array} \\ \begin{array}{c} 0\\ 1\\ 2\\ 3\\ 4\\ 5\\ 6\\ 7\\ 8\\ 9\end{array}&\left[\begin{array}{cccccccccccccccc}0&0&0&0&0&0&0&0&0&0&0&0&0&0&0&\cellcolor[gray]{0.8}0\\ 0&1&2&3&4&5&6&7&8&\cellcolor[gray]{0.8}9&\cellcolor[gray]{0.8}9&\cellcolor[gray]{0.8}9&\cellcolor[gray]{0.8}9&\cellcolor[gray]{0.8}9&\cellcolor[gray]{0.8}9&9\\ 0&2&4&6&8&\cellcolor[gray]{0.8}9&\cellcolor[gray]{0.8}10&\cellcolor[gray]{0.8}11&\cellcolor[gray]{0.8}12&13&14&15&16&17&18&19\\ 0&2&5&\cellcolor[gray]{0.8}8&\cellcolor[gray]{0.8}9&10&12&14&16&18&20&22&24&26&28&30\\ 0&2&6&8&10&12&15&18&21&24&27&30&33&36&39&42\\ 0&2&\cellcolor[gray]{0.8}6&8&11&14&18&22&26&30&34&38&42&46&50&54\\ 0&2&6&9&13&17&22&27&32&37&42&47&52&57&62&67\\ 0&\cellcolor[gray]{0.8}2&6&10&15&20&26&32&38&44&50&56&62&68&74&80\\ 0&2&7&12&18&24&31&38&45&52&59&66&73&80&87&94\\ \cellcolor[gray]{0.8}0&2&8&14&21&28&36&44&52&60&68&76&84&92&100&108\end{array}\right]\end{array}.$$
The grey entries are the first entries, in each column, that are equal in matrices $T$ and $\Phi$.

So, by Proposition \ref{t} $$ \tilde{\rho}_1(R,S)=1*3+3=6,$$
$$ \tilde{\rho}_2(R,S)=2*3+3=2*2+5=9,$$
$$ \tilde{\rho}_3(R,S)=3*2+5=11,$$
$$ \tilde{\rho}_4(R,S)=4*2+5=4*1+9=13,$$
$$ \tilde{\rho}_5(R,S)=5*1+9=14,$$
$$ \tilde{\rho}_6(R,S)=6*1+9=6*0+15=15.$$

Suppose there is a matrix $A\in {\cal A}(R,S)$ such that $\rho_i (A)= \tilde{\rho}_i(R,S),$ for $i=1,2,3,4,5,6$.  Since $ \tilde{\rho}_5(R,S)=5*1+9=14,$ by Theorem \ref{t3}, the 1's of $A$ can be  covered by one row and nine columns. Using the fact that all columns of $A$ has at least a nonzero entry, and $A$ has five columns with at least two nonzero entries, then with one row we must cover all 1's of six columns of $A$. The unique row of $A$ with at least $6$ nonzero entries is the first. Consequently, we can assume that the first row of $A$ is

 $$\left[\begin{array}{ccccccccccccccc}0&0&0&0&0&0&0&0&0&1&1&1&1&1&1\end{array}\right].$$

Let $C$ be the matrix obtained from $A$ removing the first row and the last six columns. Since $ \tilde{\rho}_3(R,S)=3*2+5=11,$ then the 1's of $A$ can be covered with two rows and five columns. Using the above arguments,  the 1's of $C$ can be covered with one row and five columns. Note that $C$ has nine columns and five of them has at least two nonzero entries. This implies that the row that cover the 1's of $C$, cover the 1's in the last four columns of $C$. Only the first and the second rows of $C$ has at least four nonzero entries. Let $D$ be the matrix obtained from $C$ removing  the last four columns.

The matrix $D$ has five columns and eight rows. Moreover, the row sum vector of $D$ is $G_1=(1,4,3,3,2,2,1,1)$ or $G_2=(5,0,3,3,2,2,1,1)$, and the column sum vector of $D$ is $H=(7,3,3,2,2)$.

Since $ \tilde{\rho}_1(R,S)=1*3+3=6,$ then the 1's of  $A$ can be covered with three rows and three columns. We have two cases:

If the row sum vector of $D$ is $G_2=(5,0,3,3,2,2,1,1)$ and its column sum is $H=(7,3,3,2,2)$, using the above arguments, all 1's of $D$ will be covered with one row and three columns. This is impossible because  five columns of $D$ has at least two nonzero entries.

If the row sum vector of $D$ is $G_1=(1,4,3,3,2,2,1,1)$,  and its column sum is $H=(7,3,3,2,2)$, using the above arguments, all 1's of $D$ will be covered with two rows (the first row of $D$ is one of these two rows) and three columns. This is impossible because four columns of the matrix obtained from $D$ removing the first row has at least two nonzero entries.

Therefore, no matrix in ${\cal A}(R,S)$ realizes all  the minimum $i$-term ranks, for $i\geq 1$.

 \vspace{0,2cm}

\section{Special partitions $R$ and $S$}

\hspace{3ex} Let $R=(R_1,\ldots ,R_m)$ and  $S=(S_1,\ldots ,S_n)$ be partitions of the same weight such that $\mathcal{A}(R,S)$ is nonempty. In this section we present conditions for the existence of a matrix in $\mathcal{A}(R,S)$ that realizes all  the minimum $t$-term ranks, for  $t\geq 1$. So the next result is our main result:

 \begin{theorem} \label{te-1} Let $R$ and $S$ be be two partitions of the same weight such that $R=(R_1,\ldots ,R_m)$, $S=(S_1,\ldots ,S_n)$, with $m>2$, $n>2$,  and $\mathcal{A}(R,S)$ is nonempty. Let $t$ be a positive integer. If the minimum integers
	$f,f^{\prime}$ such that $\phi_{2,f}=t_{2,f}$ and $\phi_{1,f^{\prime}}=t_{1,f^{\prime}}$ verify
	$1\leq f<f^{\prime}< n$, $S_f=S_{f+1}=\ldots =S_n=1$, and for all $1\leq k\leq t$,
	$\tilde{\rho}_k(R,S)\in\{k+f^{\prime},2k+f\}$, then there exists a matrix $A \in \mathcal{A}(R,S)$ such that $$\rho_k (A)=\widetilde{\rho}_k(R,S),\hspace{6ex}\mbox{ for }\hspace{6ex}k=1,\ldots,t.$$
	
\end{theorem}

\vspace{0,2cm}

For proving this theorem we start the following lemma:

\vspace{0,15cm}

\begin{lemma}{\rm \cite{B2}} \label{l1} Let $(R_1,\ldots, R_m)$ and $(S_1,\ldots,S_n)$  be two partitions of the same weight, and let $[c_{i,j}]$ be an $m$-by-$n$ nonnegative integral matrix. There is  a nonnegative integral matrix $A=[a_{i,j}]$, satisfying
	$$0\leq a_{i,j} \leq c_{i,j},\,\,\text{ for }i\in\{1,\ldots,m\} \text{ and }\,\,\,j\in\{1,\ldots,n\},$$
	$$\sum_{j=1}^n a_{i,j}=R_i,\,\,\,\,\,\,\text{ for }i\in\{1,\ldots,m\},$$
	and
	$$\sum_{i=1}^n a_{i,j}=S_j,\,\,\,\,\,\,\text{ for }j\in\{1,\ldots,n\},$$
	if and only if, for all $I\subseteq \{1,\ldots,m\}$ and $J\subseteq \{1,\ldots,n\}$,
	$$\sum_{i\in I, j\in J} c_{i,j}\geq \sum_{j\in J}S_j-\sum_{i \notin I}R_i.$$ 
\end{lemma}

 Let $T=[t_{i,j}]$  be  the $(m+1)$-by-$(n+1)$ structure matrix associated with $R$ and $S$. For $0\leq a< b\leq m$, and $0\leq c<d\leq n$, define the nonnegative integer  $\psi_{a,b;c,d}$ as
$$\begin{array}{lll}
\psi_{a,b;c,d}&=&\min \{t_{i_1,d+j_3}+t_{a+i_2,c+j_2}+t_{b+i_3,j_1}+(a-i_1)(d-c-j_2)+\\
 & &\,\,\,\,\,\,\,\,\,\,\,\, +(b-a-i_2)(c-j_1)+(a-i_1)(c-j_1)\},
\end{array}$$
where the minimum is taken  over all integers   $i_1,i_2,i_3$ and $j_1,j_2,j_3$  that satisfy
$$0\leq i_1\leq a\leq a+i_2\leq b\leq b+i_3\leq m$$ and  $$0\leq j_1\leq c\leq c+j_2\leq d\leq  d+j_3\leq n.$$

To compute a matrix in the conditions of  Theorem \ref{t3} we can use the modified Ryser's algorithm (see \cite{B-4} or Section 6) or the known algorithms from the network flows (see \cite{ford}). The following proposition is a generalization of Theorem \ref{t3}.

\begin{proposition} \label{po-1} Let $R$ and $S$ be two partitions of the same weight such that $R=(R_1,\ldots ,R_m)$, $S=(S_1,\ldots ,S_n)$, and $\mathcal{A}(R,S)$ is nonempty. Let $e,e^{\prime},f,f^{\prime}$ be integers such that $0\leq e^{\prime}< e\leq m$ and $0\leq f<f^{\prime}\leq n$. Then
 there is a matrix in $\mathcal{A}(R,S)$ of the form
\begin{equation}\label{M1}\left[
    \begin{array}{ccc}
      X & Y & V\\
      Z & W& O_{e-e^{\prime},n-f^{\prime}} \\
       U &O_{m-e,f^{\prime}-f} & O_{m-e,n-f^{\prime}}
    \end{array}
  \right].
  \end{equation}
  if and only if $$\psi_{e^{\prime},e;f,f^{\prime}}\geq t_{e,f}+t_{e^{\prime},f^{\prime}}.$$
\end{proposition}

\begin{proof} The proof follows the steps of the proof of theorem 3.5.8 in \cite{B3}.  Let $C$ be the $m$-by-$n$ matrix such that
$$C=[c_{i,j}]=\left[
                \begin{array}{ccc}
                  J_{e^{\prime},f} & J_{e^{\prime},f^{\prime}-f} & J_{e^{\prime},n-f^{\prime}}\\
                  J_{e-e^{\prime},f}  &J_{e-e^{\prime},f^{\prime}-f} & O_{e-e^{\prime},n-f^{\prime}}  \\
                 J_{m-e,f} &O_{m-e,f^{\prime}-f} & O_{m-e,n-f^{\prime}}
                \end{array}
              \right],
$$
where $J_{u,v}$ denote the $u$-by-$v$ matrix whose entries are all equal to $1$. By Lemma \ref{l1}, there is a matrix in $\mathcal{A}(R,S)$ of the form (\ref{M1}) if and only if
\begin{equation} \label{i11}
\sum_{i\in I, j\in J} c_{i,j}\geq \sum_{j\in J}S_j-\sum_{i \notin I}R_i,
\end{equation}
 for all $I\subseteq \{1,\ldots,m\}$ and $J\subseteq \{1,\ldots,n\}$.

 For $I\subseteq \{1,\ldots,m\}$ and $J\subseteq \{1,\ldots,n\}$, we write $I=I_1 \cup I_2\cup I_3$ where  $I_1\subseteq \{1,\ldots,e^{\prime}\}$, $I_2\subseteq \{e^{\prime}+1,\ldots,e\}$ and $I_3\subseteq \{e+1,\ldots,m\}$, and we write $J=J_1 \cup J_2\cup J_3$ where  $J_1\subseteq \{1,\ldots,f\}$, $J_2\subseteq \{f+1,\ldots,f^{\prime}\}$ and $J_3\subseteq \{f^{\prime}+1,\ldots,n\}$. We agree to take  complements of $I_1$, $I_2$, $I_3$, $J_1$, $J_2$, and $J_3$ with respect to $\{1,\ldots,e^{\prime}\}$, $\{e^{\prime}+1,\ldots,e\}$,  $\{e+1,\ldots,m\}$, $\{1,\ldots,f\}$, $\{f+1,\ldots,f^{\prime}\}$, and $\{f^{\prime}+1,\ldots,n\}$, respectively. Then (\ref{i11}) is equivalent to
$$
|I_1||J_1|+|I_1||J_2|+|I_1||J_3|+|I_2||J_1|+|I_2||J_2|+|I_3||J_1|\geq$$
 $$\geq \sum_{j\in J_1\cup J_2\cup J_3}S_j-\sum_{i \in \overline{I_1}\cup\overline{I_2}\cup\overline{I_3}}R_i,
$$
for all $I_1\subseteq \{1,\ldots,e^{\prime}\}$, $I_2\subseteq  \{e^{\prime}+1,\ldots,e\}$, $I_3\subseteq  \{e+1,\ldots,m\}$,  $J_1\subseteq \{1,\ldots,f\}$,  $J_2\subseteq \{f+1,\ldots,f^{\prime}\}$ and $J_3\subseteq  \{f^{\prime}+1,\ldots,n\}$. Let $k_1=|I_1|$, $k_2 =|I_2|$, $k_3 =|I_3|$, $l_1=|J_1|$, $l_2=|J_2|$ and  $l_3=|J_3|$. Since $R$ and $S$ are nonincreasing, it follows that the last inequality is equivalent to
$$k_1l_1+k_1l_2+k_1l_3+k_2l_1+k_2l_2+k_3l_1\geq $$
 $$\geq \sum_{j=1}^{l_1}S_j + \sum_{j=f+1}^{f+l_2}S_j + \sum_{j=f^{\prime}+1}^{f^{\prime}+l_3}S_j-  \sum_{i=e^{\prime}-k_1+1}^{e^{\prime}}R_i-  \sum_{i=e-k_2+1}^{e}R_i - \sum_{i=m-k_3+1}^{m}R_i,$$
holding for all integers $k_1,k_2$, and $k_3$, with $0\leq k_1\leq e^{\prime}\leq e^{\prime}+k_2\leq e\leq e+k_3\leq m$ and all integers $l_1,l_2$, and $l_3$, with $0\leq l_1\leq f\leq f+l_2\leq f^{\prime}\leq f^{\prime}+l_3\leq n$. The last inequality is equivalent to
\begin{equation} \label{i31}
t_{K,L}-k_2l_3-k_3l_2-k_3l_3\geq 0,
\end{equation}
 where $K_1=\{1,\ldots,k_1\}$, $K_2=\{e^{\prime}+1,\ldots,e^{\prime}+k_2\}$, $K_3=\{e+1,\ldots,e+k_3\}$, $K=K_1\cup K_2\cup K_3$,  $L_1=\{1,\ldots,l_1\}$, $L_2=\{f+1,\ldots,f+l_2\}$, $L_3=\{f^{\prime}+1,\ldots,f^{\prime}+l_3\}$, $L=L_1\cup L_2\cup L_3$. Let $A$ be any matrix  in $\mathcal{A}(R,S)$, and  partition  $A$ according to the diagram
$$\begin{array}{c|c|c|c|c|c|c}
        &L_1 & \overline{L_1} & L_2 & \overline{L_2}& L_3 & \overline{L_3} \\ \hline
        K_1 &(0) & &(0) & &(0) & \\ \hline
        \overline{K_1} & &(1) & &(1) & &(1) \\ \hline
        K_2 &(0) & & (0)& &-(1) & \\ \hline
        \overline{K_2} & &(1) & &(1) & &(1) \\ \hline
        K_3 &(0) & &-(1) & & -(1)& \\ \hline
        \overline{K_3} & &(1) & &(1) & &(1)
\end{array},
$$ Then (\ref{i31}) counts $0^{\prime}$ and $1^{\prime}$ in submatrices of $A$ as shown above, where $-(1)$ means the negative of the number of $1^{\prime}$ in the submatrices of $A$ indicated.

On the other hand, the expression $t_{k_1,f^{\prime}+l_3}+t_{e^{\prime}+k_2,f+l_2}+t_{e+k_3,l_1}$ counts $0^{\prime}$ and $1^{\prime}$  as shown below:

$$\begin{array}{c|c|c|c|c|c|c}
        &L_1 & \overline{L_1} & L_2 & \overline{L_2}& L_3 & \overline{L_3} \\ \hline
        K_1 &3(0) &2(0) &2(0) &(0) &(0) & \\ \hline
        \overline{K_1} &2(0) &(0) &(0) & & &(1) \\ \hline
        K_2 &2(0) &(0) & (0)& & &(1) \\ \hline
        \overline{K_2} & (0)& & &(1) & (1)&2(1) \\ \hline
        K_3 &(0) & & &(1) & (1)&2(1) \\ \hline
        \overline{K_3} & &(1) &(1) &2(1) &2(1) &3(1)
\end{array},
$$
In the matrix, $2(0)$ means twice the number of  $0^{\prime}$ in the corresponding submatrices of $A$, $3(0)$ means  three times the number of  $0^{\prime}$ in the corresponding submatrices of $A$, $2(1)$ means twice the number of  $1^{\prime}$ in the corresponding submatrices of $A$, and $3(1)$ means three times the number of  $1^{\prime}$ in the corresponding submatrices of $A$.

It now follows that
\begin{equation} \label{i41}
t_{k_1,f^{\prime}+l_3}+t_{e^{\prime}+k_2,f+l_2}+t_{e+k_3,l_1}-(t_{K,L}-k_2l_3-k_3l_2-k_3l_3)
\end{equation}
counts $0^{\prime}$ and $1^{\prime}$ in submatrices of $A$ as indicated in the matrix diagram.

$$\begin{array}{c|c|c|c|c|c|c}
        &L_1 & \overline{L_1} & L_2 & \overline{L_2}& L_3 & \overline{L_3} \\ \hline
        K_1 &2(0) &2(0) &(0) &(0) & & \\ \hline
        \overline{K_1} &2(0) &(0)-(1) &(0) &-(1) & & \\ \hline
        K_2 &(0) &(0) & & &(1) &(1) \\ \hline
        \overline{K_2} &(0)&-(1)& & & (1)&(1) \\ \hline
        K_3 & & &(1) &(1) & 2(1)&2(1) \\ \hline
        \overline{K_3} & & &(1) &(1) &2(1) &2(1)
\end{array},
$$
Hence, (\ref{i41}) equals $$t_{e,f}+t_{e^{\prime},f^{\prime}}-(e^{\prime}-k_1)(f-l_1)-(e^{\prime}-k_1)(f^{\prime}-f-l_2)-(e-e^{\prime}-k_2)(f-l_1).$$

 Therefore (\ref{i31}) holds if and only if $\psi_{e^{\prime},e;f,f^{\prime}} \geq t_{e,f}+t_{e^{\prime},f^{\prime}}$.\hfill \end{proof}

  \begin{proposition} \label{po-2} Let $R$ and $S$ be be two partitions of the same weight such that $R=(R_1,\ldots ,R_m)$, $S=(S_1,\ldots ,S_n)$, with $m>2$, $n>2$, and $\mathcal{A}(R,S)$ is nonempty. Let $f,f^{\prime}$ be integers such that  $1\leq f<f^{\prime}< n$ and $S_f=S_{f+1}=\ldots =S_n=1$. If $\phi_{2,f}=t_{2,f}$ and $\phi_{1,f^{\prime}}=t_{1,f^{\prime}}$, then $$\psi_{1,2;f,f^{\prime}}\geq t_{1,f^{\prime}}+t_{2,f}.$$
\end{proposition}

\begin{proof} Let $p$ be an integer such that $f\leq p\leq n$. Then, $$t_{1,p}=p+\sum_{i=2}^mR_i-\sum_{j=1}^pS_j=t_{1,f}+(p-f)+\sum_{j=f+1}^pS_j.$$ Since $S_f=S_{f+1}=\ldots =S_n=1$ then $$t_{1,p}=t_{1,f}.$$
	
	\vspace{1ex}
	
	Let $s$ be an integer such that $f\leq s \leq n$. Then, $$t_{2,s}=2s+\sum_{i=3}^mR_i-\sum_{j=1}^sS_j=t_{2,f}+2(s-f)+\sum_{j=f+1}^sS_j.$$ Since $S_f=S_{f+1}=\ldots =S_n=1$ then $$t_{2,s}=t_{2,f}+(s-f).$$
	
	\vspace{1ex}
	
	By definition, $$\psi_{1,2;f,f^{\prime}}=\min\{t_{i_1,f^{\prime}+j_3}+t_{1+i_2,f+j_2}+t_{2+i_3,j_1}+$$ $$+(1-i_1)(f^{\prime}-f-j_2)+(2-1-i_2)(f-j_1)+(1-i_1)(f-j_1)\},$$
	where the minimum is taken  over all integers   $i_1,i_2,i_3$ and $j_1,j_2,j_3$  that satisfy
	$$0\leq i_1\leq 1\leq 1+i_2\leq 2\leq 2+i_3\leq m$$ and  $$0\leq j_1\leq f\leq f+j_2\leq f^{\prime}\leq  f^{\prime}+j_3\leq n.$$

	Therefore, we may conclude that:
	\begin{itemize}
		\item If  $i_1=0$, then $t_{i_1,f^{\prime}+j_3}=t_{0,n}=0$;
		\item If  $i_1=1$, then $t_{i_1,f^{\prime}+j_3}=t_{1,f}$;
		\item If  $i_2=0$, then  $t_{1+i_2,f+j_2}=t_{1,f}$;
		\item If  $i_2=1$, then  $t_{1+i_2,f+j_2}=t_{2,f}+ j_2$.
	\end{itemize}
	
	So, we have four cases:
	\begin{itemize}
		\item {\bf Case 1} If  $i_1=0$ and $i_2=0$, then
	$$\psi_{1,2;f,f^{\prime}}=0+t_{1,f^{\prime}}+t_{2+i_3,j_1}+0+(f-j_1)+(f-j_1).$$

 Since $t_{2,f}=\phi_{2,f}\leq 0+t_{2+i_3,j_1}+2(f-j_1)$, we get $$\psi_{1,2;f,f^{\prime}}\geq t_{1,f^{\prime}}+t_{2,f}.$$

		\item {\bf Case 2} If  $i_1=0$ and $i_2=1$, then $$\psi_{1,2;f,f^{\prime}}=0+t_{2,f}+j_2+t_{2+i_3,j_1}+(f^{\prime}-f-j_2)+0+(f-j_1).$$

		Since $t_{1,f^{\prime}}=\phi_{1,f^{\prime}}\leq 0+t_{2+i_3,j_1}+(f^{\prime}-j_1)$, we get
		
		 $$\psi_{1,2;f,f^{\prime}}\geq t_{1,f^{\prime}}+t_{2,f}.$$

		\item {\bf Case 3} If  $i_1=1$ and $i_2=0$, then $$\psi_{1,2;f,f^{\prime}}=t_{1,f^{\prime}}+t_{1,f^{\prime}}+t_{2+i_3,j_1}+0+(f-j_1)+0.$$
		
		Since $t_{2,f}=\phi_{2,f}\leq t_{1,f^{\prime}}+t_{2+i_3,j_1}+(f-j_1)$, we get
		
		$$\psi_{1,2;f,f^{\prime}}\geq t_{1,f^{\prime}}+t_{2,f}.$$

		\item {\bf Case 4} If  $i_1=1$ and $i_2=1$, then
		$$\psi_{1,2;f,f^{\prime}}=t_{1,f^{\prime}}+t_{2,f}+j_2+t_{2+i_3,j_1}+0+0+0.$$
		
		Consequently,
		
		$$\psi_{1,2;f,f^{\prime}}\geq t_{1,f^{\prime}}+t_{2,f}.$$
	\end{itemize}

\hfill\end{proof}

 \begin{lemma} \label{li-1} Let $e,\ e^{\prime}, \ f,\ f^{\prime},\ k,$ and $l$ be nonnegative integers such that $1\leq k<l$, $ke+f<ke^{\prime}+f^{\prime}$, and $le+f>le^{\prime}+f^{\prime}$. Then $e^{\prime}<e$ and $f<f^{\prime}$.
 \end{lemma}

 \begin{proof} Since $ke+f<ke^{\prime}+f^{\prime}$, and $le+f>le^{\prime}+f^{\prime}$, we have $$k(e-e^{\prime})<f^{\prime}-f<l(e-e^{\prime}).$$
 If $e\leq e^{\prime}$ then $e-e^{\prime}\leq 0$. Using the inequality $1\leq k<l$ we get $$k(e-e^{\prime})\geq l(e-e^{\prime}).$$ Contradiction. So, $e^{\prime}<e$. Consequently, $e-e^{\prime}>0$ and $$0<k(e-e^{\prime})<f^{\prime}-f.$$ Therefore, $f<f^{\prime}$. \hfill \end{proof}

\vspace{0,3cm}

\begin{proofm} If $t=1$, the result follows. Let $t>1$. Suppose there is no  matrix $A \in \mathcal{A}(R,S)$ such that $$\rho_k (A)=\widetilde{\rho}_k(R,S),\hspace{6ex}\mbox{for } \hspace{6ex}k=1,\ldots ,t.$$ Let $l$ be the greatest integer such that $1\leq l\leq t$ and there is a matrix $D$  in $\mathcal{A}(R,S)$ such that
	$$\rho_1 (D)=\widetilde{\rho}_1(R,S),\ldots, \rho_{l-1} (D)=\widetilde{\rho}_{l-1}(R,S), \,  \text{ and } \, \rho_{l} (D)\neq \widetilde{\rho}_{l}(R,S).$$
		
		Let $B$ a matrix in $\mathcal{A}(R,S)$ such that
		$$\rho_{l} (B)=\widetilde{\rho}_{l}(R,S)<  \rho_{l} (D).$$
		
 So, there is an integer $k$, with $1\leq k<l$, such that
$$\rho_{k+1} (B)=\widetilde{\rho}_{k+1}(R,S),\ldots ,  \rho_{l} (B)=\widetilde{\rho}_{l}(R,S),$$
  and $$  \rho_{k} (B)\neq \widetilde{\rho}_{k}(R,S)=\rho_{k} (D). $$

 Consequently, $$ \rho_{k} (B)>\rho_{k} (D), \, \mbox{ and }  \, \rho_{l} (D)>\rho_{l} (B).$$

 By Proposition \ref{t}, there are nonnegative integers $a,\ b, \ c,\ d$, such that $\widetilde{\rho_{k} }(R,S)=kb+c$, $\phi_{b,c}=t_{b,c}$, and $\widetilde{\rho_{l} }(R,S)=la+d$, $\phi_{a,d}=t_{a,d}$.

		Then $1\leq k<l$, $kb+c<ka+d$, and $lb+c>la+d$.
		Using Lemma \ref{li-1}, we get  $a<b$ and $c<d$. By hypothesis, $a=1$, $b=2$, $c=f$ and $d=f^{\prime}$.
		
		Using Proposition \ref{po-2}, we have $$\psi_{1,2;f,f^{\prime}}\geq t_{1,f^{\prime}}+t_{2,f}.$$ By Proposition \ref{po-1}, we conclude that there exists a matrix $Q$ in $\mathcal{A}(R,S)$ such that all $1^{\prime}$ of $Q$ are contained in the union of the first $2$ rows and first $f$ columns, and  all $1^{\prime}$ of $Q$ are contained in the union of the first  row and first $f^{\prime}$ columns. Thus, by Proposition \ref{te+f} $$\rho_1 (Q)=\widetilde{\rho_1}(R,S), \ldots,  \rho_{l} (Q)=\widetilde{\rho}_{l}(R,S).$$ Contradiction. Therefore,  there exists a matrix $A \in \mathcal{A}(R,S)$ such that $\rho_k (A)=\widetilde{\rho}_k(R,S),\mbox{ for all }k=1,\ldots ,t.$

	 \hfill\end{proofm}

\section{Algorithm for construct matrices with fixed zero blocks \label{algo}}

\hspace{3ex} Let $R=(R_1,\ldots , R_n)$ and $S=(S_1,\ldots ,S_n)$ be partitions of the same weight. Let $e,e',f$ and $f'$ positive integers such that  $0\leq e < e'\leq m$ and $0\leq   f^{\prime}<f\leq n$. Denote by $\mathcal{A}_{e,e^{\prime};f^{\prime}f}(R,S)$ the set of all matrices of $\mathcal{A}(R,S)$ whose all 1's  are covered by $e$ rows and $f$ columns and, by $e^{\prime}$ rows and $f^{\prime}$ columns. Assume that $\mathcal{A}_{e,e^{\prime};f^{\prime}f}(R,S)$ is nonempty. In this section we present an algorithm for construct a matrix in  $\mathcal{A}_{e,e^{\prime};f^{\prime}f}(R,S)$. This algorithm  generalizes the algorithm stated by Brualdi and Dahl in \cite{B-4} for construct a matrix in $\mathcal{A}_{e,f}(R,S)$ (the subset of all matrices of $\mathcal{A}(R,S)$ whose all 1's  are covered by $e$ rows and $f$ columns).

First we describe the  modified Ryser algorithm.

\vspace{2ex}

Let $R^{(e)}=(R_1,\ldots ,R_e)$, $S^{(f)}=(S_1,\ldots ,S_f)$.

Let $F=(F_1,\ldots ,F_p)$ be an integer vector with $p$-coordinates. Let $\sigma_F$ be a permutation of $\{1,\ldots ,p\}$ such that the vector $F_{\sigma}=(F_{\sigma(1)},\ldots ,F_{\sigma(p)})$ verify $F_{\sigma(1)}\geq \ldots \geq F_{\sigma(p)}$. We denote by $P_{\sigma_F}$ the $p$-by-$p$ permutation matrix associated with $\sigma_F$, and by $P_{\sigma_F}^{-1}$ we denote its inverse.

\vspace{3ex}

{\bf The modified Ryser algorithm:}
\begin{enumerate}
	\item Start with an $m$-by-$n$ $(0,1)$-matrix $B_m$ whose row sum vector is $R$ and whose column sum vector is  $R^*$. Thus the 1's occupy the initial positions in each row. Let $B_{e,n}$ be the submatrix obtained from $B_m$ deleting rows $e+1,\ldots ,m$. Let $\overline{B}_0$ be the $e$-by-$0$ empty matrix.
	\item  For $k=n,\ n-1,\ldots ,f+1$, do:
	
	Shift to column $k$ the final $1$'s in those $S_k$ rows of $B_{e,k}$ with the largest sum, with preference given to the lowest rows (those with the largest index) in case of ties. This results in a matrix $$\left[\begin{array}{c|c}B_{e,k-1}&\overline{B}_{n-k+1}\end{array}\right],$$ where $B_{e,k-1}$ has $k-1$ columns.
	
	\item Let $C_n$ be the  $n$-by-$m$ $(0,1)$-matrix  whose row sum vector is $S$ and whose column sum vector is  $S^*$. Thus the 1's occupy the initial positions in each row. Let $C_{f,m}$ be the submatrix obtained from $C_n$ deleting rows $f+1,\ldots ,n$. Let $\overline{C}_0$ be the $f$-by-$0$ empty matrix.

	\item For $h=m,\ m-1,\ldots ,e+1$, do:
	
	Shift to column $h$ the final $1$'s in those $R_h$ rows of $C_{f,h}$ with the largest sum, with preference given to the lowest rows (those with the largest index) in case of ties. This results in a matrix $$\left[\begin{array}{c|c}C_{f,h-1}&\overline{C}_{m-h+1}\end{array}\right],$$ where $C_{f,h-1}$ has $h-1$ columns.
	
	\item Let $\hat{R}$ be the row-sum sequence of $\overline{B}_{n-f}$ and $\hat{S}$ be the row-sum sequence of $\overline{C}_{m-e}$. Let $\bar{R}=R^{(e)}-\hat{R}$ and $\bar{S}=S^{(f)}-\hat{S}$.
	
	\item Let  $A_1$ be the canonical matrix of ${\cal A}(\bar{R}_{\sigma_{\bar{R}}},\bar{S}_{\sigma_{\bar{S}}})$.
\end{enumerate}

{\bf Output:} $$A=\left[\begin{array}{c|c} & \\ P^{-1}_{\sigma_{\bar{R}}}A_1P^{-1}_{\sigma_{\bar{S}}}&\overline{B}_{n-f}\\ & \\ \hline & \\ \overline{C}^T_{m-e}& 0 \\ &  \end{array}\right],$$ where $\overline{C}^T_{m-e}$ is the transpose of $\overline{C}_{m-e}$.

\vspace{3ex}

The matrix $\overline{B}_{n-f}$ is the canonical column $f$-submatrix relative to $R$, $S$ and $f$.  The matrix $\overline{C}_{m-e}$ is the canonical column $e$-submatrix relative to $S$, $R$ and $e$.

\vspace{3ex}

\begin{example} Let $R=(4,2,2,2,1,1,1)$ and $S=(2,2,2,2,1,1,1,1,1)$. Then $R^*=(7,4,1,1)$. It is possible to prove that there is a matrix $A\in \mathcal{A}(R,S)$ whose all 1's are covered by $2$ rows and $4$ columns.

	Using last algorithm let $$B_7=\left[\begin{array}{ccccccccc}1&1&1&1&0&0&0&0&0\\ 1&1&0&0&0&0&0&0&0\\ 1&1&0&0&0&0&0&0&0\\ 1&1&0&0&0&0&0&0&0\\  1&0&0&0&0&0&0&0&0\\ 1&0&0&0&0&0&0&0&0\\ 1&0&0&0&0&0&0&0&0\end{array}\right]$$ and $$B_{2,9}=\left[\begin{array}{ccccccccc}1&1&1&1&0&0&0&0&0\\ 1&1&0&0&0&0&0&0&0\end{array}\right].$$
	
	The following matrices $\left[\begin{array}{c|c}B_{2,i}&\overline{B}_{9-i}\end{array}\right]$ are produced using step {\rm 2.}.
	$$\left[\begin{array}{ccccccccc|}1&1&1&1&0&0&0&0&0\\ 1&1&0&0&0&0&0&0&0\end{array}\right]\ ,\ \left[\begin{array}{cccccccc|c}1&1&1&0&0&0&0&0&1\\ 1&1&0&0&0&0&0&0&0\end{array}\right]\ ,$$ $$\left[\begin{array}{ccccccc|cc}1&1&0&0&0&0&0&1&1\\ 1&1&0&0&0&0&0&0&0\end{array}\right]\ , \ \left[\begin{array}{cccccc|ccc}1&1&0&0&0&0&0&1&1\\ 1&0&0&0&0&0&1&0&0\end{array}\right]\ ,$$ $$\left[\begin{array}{ccccc|cccc}1&0&0&0&0&1&0&1&1\\ 1&0&0&0&0&0&1&0&0\end{array}\right]\ , \ \left[\begin{array}{cccc|ccccc}1&0&0&0&0&1&0&1&1\\ 0&0&0&0&1&0&1&0&0\end{array}\right].$$
	
	In step {\rm 3.} $$C_{4,7}=\left[\begin{array}{ccccccc}1&1&0&0&0&0&0\\ 1&1&0&0&0&0&0\\ 1&1&0&0&0&0&0\\ 1&1&0&0&0&0&0\end{array}\right].$$
	
	The following matrices $\left[\begin{array}{c|c}C_{4,i}&\overline{C}_{7-i}\end{array}\right]$ are produced using step {\rm 4.}.

	$$\left[\begin{array}{ccccccc|}1&1&0&0&0&0&0\\ 1&1&0&0&0&0&0\\ 1&1&0&0&0&0&0\\ 1&1&0&0&0&0&0\end{array}\right]\ ,\ \left[\begin{array}{cccccc|c}1&1&0&0&0&0&0\\ 1&1&0&0&0&0&0\\ 1&1&0&0&0&0&0\\ 1&0&0&0&0&0&1\end{array}\right]\ ,$$
	$$\left[\begin{array}{ccccc|cc}1&1&0&0&0&0&0\\ 1&1&0&0&0&0&0\\ 1&0&0&0&0&1&0\\ 1&0&0&0&0&0&1\end{array}\right]\ ,\ \left[\begin{array}{cccc|ccc}1&1&0&0&0&0&0\\ 1&0&0&0&1&0&0\\ 1&0&0&0&0&1&0\\ 1&0&0&0&0&0&1\end{array}\right]\ ,$$
	$$\left[\begin{array}{ccc|cccc}1&0&0&1&0&0&0\\ 1&0&0&0&1&0&0\\ 1&0&0&0&0&1&0\\ 0&0&0&1&0&0&1\end{array}\right]\ ,\ \left[\begin{array}{cc|ccccc}1&0&0&1&0&0&0\\ 0&0&1&0&1&0&0\\ 0&0&1&0&0&1&0\\ 0&0&0&1&0&0&1\end{array}\right]$$
	
	In this case, in step {\rm 5.}, $$\overline{B}_{5}= \left[\begin{array}{ccccc}0&1&0&1&1\\ 1&0&1&0&0\end{array}\right],$$ with row sum sequence $\hat{R}=(3,2)$,
	$$\overline{C}_{5}= \left[\begin{array}{ccccc}0&1&0&0&0\\ 1&0&1&0&0\\ 1&0&0&1&0\\ 0&1&0&0&1\end{array}\right],$$ with row sum sequence $\hat{S}=(1,2,2,2)$.
	
	So, $\bar{R}=R^{(e)}-\hat{R}=(4,2)-(3,2)=(1,0)$ and $\bar{S}=S^{(f)}-\hat{S}=(2,2,2,2)-(1,2,2,2)=(1,0,0,0)$
	
	A canonical matrix of ${\cal A}(\bar{R},\bar{S})$ is $$A_1=\left[\begin{array}{cccc}1&0&0&0\\ 0&0&0&0\end{array}\right].$$
	
	Therefore, the final matrix is $$A=\left[\begin{array}{cccc|ccccc}1&0&0&0&0&1&0&1&1\\ 0&0&0&0&1&0&1&0&0 \\ \hline 0&1&1&0&0&0&0&0&0\\ 1&0&0&1&0&0&0&0&0\\ 0&1&0&0&0&0&0&0&0\\ 0&0&1&0&0&0&0&0&0\\ 0&0&0&1&0&0&0&0&0\end{array}\right].$$
\end{example}

\begin{theorem}\label{t12} Let $R=(R_1,\ldots ,R_m)$ and $S=(S_1,\ldots ,S_n)$ be partitions of the same weight. Let $e,e',f$ and $f'$ positive integers with  $0\leq e < e'\leq m$, $0\leq   f^{\prime}<f\leq n$, and ${\cal A}_{e,e';f',f}(R,S)\neq \emptyset$. Then there is a matrix
	$$A=\left[
	\begin{array}{ccc}
	A_{11} & A_{12} &A_{f} \\
	A_{21} & A_{22} & 0 \\
	A^T_{e'} & 0 & 0 \\
	\end{array}
	\right],
	$$
	in the class ${\cal A}_{e,e';f',f}(R,S)$, where $A_{f}$ is the  canonical column $f$-submatrix relative to $R$, $S$, and $f$, and $A_{e'} $ is the canonical column $e'$-submatrix  relative to $S$, $R$ and $e'$,  $A_{e'} ^T$ is its transpose, and there are permutation matrices $P$ and $Q$ such that
	$$P\left[
	\begin{array}{cc}
	A_{11} & A_{12}\\
	A_{21} & A_{22}
	\end{array}
	\right]Q$$
	is the canonical matrix in the class where it belongs.
	
\end{theorem}

\begin{proof} Let
	$$D=\left[
	\begin{array}{ccc}
	A_1 & A_2 & A_3 \\
	A_4 & A_5 & 0 \\
	A_6 & 0 & 0 \\
	\end{array}
	\right]$$
	be a matrix in  ${\cal A}_{e,e';f',f}(R,S)$.  From Ryser's interchange theorem we may apply interchanges to the matrix
	$$[A_1\,\,\, \,A_2\,\,\,\, A_3]$$
	to obtain the matrix
	$$[B_1 \,\,\, \,B_2\, \,\,\,A_{f}].$$

	Applying a similar argument to the matrix
	$$\left[
	\begin{array}{c}
	B_1 \\
	A_4 \\
	A_6 \\
	\end{array}
	\right]
	$$
	we obtain the matrix
	
	$$\left[
	\begin{array}{ccc}
	C_1 &  B_2 &  A_{f'} \\
	B_3 & A_5 & 0 \\
	A_{e'}^T & 0 & 0 \\
	\end{array}.
	\right].
	$$
	Consider the submatrix
	$$C=\left[
	\begin{array}{cc}
	C_1 &  B_2 \\
	B_3 & A_5
	\end{array}
	\right].$$
	
	Let $R'$ and $S'$ be the row sum vector and the column sum vector  of $C$, respectively, and let $C'$ be the canonical matrix of $\mathcal{A}(R'_{\sigma_{R'}},S'_{\sigma_{S'}})$. Then, there are permutation matrices $P_1$ and $Q_1$ such that $P_1C'Q_1$ has row sum vector $R'$ and column sum vector $S'$.  Now we replace $C$ by $C'$ and we have the desired matrix.
\end{proof}

\vspace{0,2cm}

We can now use Theorem \ref{t12} to give an algorithm to construct a matrix in ${\cal A}_{e,e';f,f'}(R,S)$:

\vspace{0,2cm}
{\bf Algorithm:}

\vspace{2ex}

Let $e,e',f$ and $f'$ positive integers such that  $0\leq e^{\prime}< e\leq m$ and $0\leq f<f^{\prime}\leq n$.
\begin{enumerate}
	\item Construct  $A_{f}$, the canonical column $f$-submatrix,  relative to $R$, $S$ and $f$.
	\item Construct $A_{e'}$, the canonical column $e'$-submatrix,  relative to $S$, $R$ and $e$.
	
	\item Let $\hat{R}$ be the row-sum sequence of $A_{f}$ and $\hat{S}$ be the row-sum sequence of $A_{e´'}$. Let $\bar{R}=R^{(e')}-\hat{R}$ and $\bar{S}=S^{(f)}-\hat{S}$.
	\item Let  $A_1$ be the canonical matrix of ${\cal A}(\bar{R}_{\sigma_{\bar{R}}},\bar{S}_{\sigma_{\bar{S}}})$.
\end{enumerate}

\vspace{3ex}

\begin{example} Let $R=(4,2,2,2,1,1,1)$ and $S=(2,2,2,2,1,1,1,1,1)$ as in last example. It is possible to prove that there is a matrix $A\in \mathcal{A}(R,S)$ whose all 1´s are  covered by $2$ rows and $4$ columns and, by $3$ rows and $3$ columns.
	
	In this case, $e=2$, $e'=3$, $f'=3$ and $f=4$.
	
	In last example we constructed the canonical column $f$-submatrix,  relative to $R$, $S$ and $f$,
	$$A_{f}=\bar{B}_5=\left[\begin{array}{ccccc}0&1&0&1&1\\ 1&0&1&0&0\end{array}\right].$$ This matrix has  row sum sequence $\hat{R}=(3,2)$.
	
	The canonical column $e'$-submatrix,  relative to $S$, $R$ and $e'$ is
	$$A_{e'}=\left[\begin{array}{cccc}0&1&0&0\\  1&0&1&0\\ 1&0&0&1 \end{array}\right].$$
	This matrix has  row sum sequence $\hat{S}=(1,2,2)$.

	So, $\bar{R}=R^{(e')}-\hat{R}=(4,2,2)-(3,2,0)=(1,0,2)$ and $\bar{S}=S^{(f)}-\hat{S}=(2,2,2,2)-(1,2,2,0)=(1,0,0,2)$
	
	A canonical matrix of ${\cal A}(\bar{R}_{\sigma_{\bar{R}}},\bar{S}_{\sigma_{\bar{S}}})={\cal A}((2,1,0),(2,1,0,0))$ is $$A_1=\left[\begin{array}{cccc}1&1&0&0\\ 1&0&0&0 \\ 0&0&0&0\end{array}\right].$$
	
	Therefore, the final matrix is $$A=\left[\begin{array}{c|c}\begin{array}{cccc}0&0&0&1\\ 0&0&0&0 \\ 1&0&0&1\end{array}& \begin{array}{ccccc}0&1&0&1&1\\ 1&0&1&0&0 \\ \hline 0&0&0&0&0\end{array}\\ \hline \begin{array}{ccc|c}0&1&1&0\\ 1&0&0&0\\ 0&1&0&0\\ 0&0&1&0\end{array}& \begin{array}{ccccc}0&0&0&0&0\\ 0&0&0&0&0\\ 0&0&0&0&0\\ 0&0&0&0&0\end{array}\end{array}\right].$$
\end{example}

\end{document}